\documentclass[11pt,reqno]{amsart}

\usepackage{amsthm,amssymb,amstext,amscd,amsfonts,amsbsy,amsxtra,latexsym,amsmath}
\usepackage{fullpage}
\usepackage[english]{babel}
\usepackage[latin1]{inputenc}
\usepackage{verbatim}
\usepackage{graphicx} 
\usepackage[all]{xy} 
\usepackage{enumitem}
\numberwithin{figure}{section} 
\allowdisplaybreaks

\newcommand{\field}[1]{\mathbb{#1}} 
\newcommand{\N}{\field{N}}
\newcommand{\Z}{\field{Z}} 
\newcommand{\R}{\field{R}}
\newcommand{\C}{\field{C}}
\newcommand{\Q}{\field{Q}}
\newcommand{\sgn}{\operatorname{sgn}}

\newcommand{\SL}{\operatorname{SL}}

\newcommand{\Res}{\operatorname{Res}}

\renewcommand{\H}{\mathbb{H}}

\newcommand{\wt}{\kappa}

\newcommand{\SSnew}{\mathbb{S}}

\numberwithin{equation}{section}
\newtheorem{theorem}{\textbf{Theorem}}
\numberwithin{theorem}{section}
\newtheorem{corollary}[theorem]{\textbf{Corollary}}
\newtheorem{lemma}[theorem]{\textbf{Lemma}}
\newtheorem{proposition}[theorem]{\textbf{Proposition}}
\theoremstyle{remark}
\newtheorem*{remark}{Remark}
\newtheorem*{remarks}{Remarks}

\renewenvironment{proof}[1][Proof]{\begin{trivlist}
\item[\hskip \labelsep {\bfseries #1:}]}{\qed\end{trivlist}}

\newcommand{\bea}{\begin{eqnarray}} 
\newcommand{\eea}{\end{eqnarray}} 
\newcommand{\be}{\begin{equation}} 
\newcommand{\ee}{\end{equation}} 
\newcommand{\benn}{\begin{equation*}} 
\newcommand{\eenn}{\end{equation*}} 
\newcommand{\arctanh}{\operatorname{artanh}}

\theoremstyle{definition}
\newtheorem*{definition}{\textbf{Definition}}

\title[]{Polar harmonic Maass forms and their applications}
\author{Kathrin Bringmann} 
\address{Mathematical Institute\\University of
Cologne\\ Weyertal 86-90 \\ 50931 Cologne \\Germany}
\email{kbringma@math.uni-koeln.de}
\author{Ben Kane}
\address{Department of Mathematics\\ University of Hong Kong\\ Pokfulam, Hong Kong}
\email{bkane@maths.hku.hk}
\date{\today}
\thanks{ The research of the first author was supported by the Alfried Krupp Prize for Young University Teachers of the Krupp foundation and the research leading to these results has received funding from the European Research Council under the European Union's Seventh Framework Programme (FP/2007-2013) / ERC Grant agreement n. 335220 - AQSER.  The research of the second author was supported by grant project numbers 27300314 and 17302515 of the Research Grants Council.}
\begin{document}
\begin{abstract}
In this survey, we present recent results of the authors about non-meromorphic modular objects known as polar harmonic Maass forms.  These include the computation of Fourier coefficients of meromorphic modular forms and relations between inner products of meromorphic modular forms and higher Green's functions evaluated at CM-points. 
\end{abstract}

\maketitle

\section{Introduction}

While investigating the Doi-Naganuma lift, Zagier \cite{ZagierRQ} encountered interesting weight $2k$ cusp forms ($z\in\mathbb{H}$, $k\in \N_{\geq 2}$, $\delta\in\N$)
\begin{equation}\label{Zagiercusp}
f_{k,\delta}(z):=\sum_{\mathcal{Q}\in \mathcal{Q}_{\delta}} \mathcal{Q}(z,1)^{-k},
\end{equation}
where $\mathcal{Q}_{\delta}$ is the set of integral binary quadratic forms of discriminant $\delta$.  These cusp forms were then used by Kohnen and Zagier \cite{KohnenZagier} to construct a kernel function for the Shimura and Shintani lifts between integral and half-integral weight cusp forms.  Integrating the resulting theta function against half-integral weight (parabolic) Poincar\'e series then yields the functions $f_{k,\delta}$ as theta lifts.  Kramer \cite{Kramer} showed that related cusp forms $f_{k,\delta,[Q]}$, defined by restricting the sum in \eqref{Zagiercusp} to those $\mathcal{Q}$ in a fixed $\SL_2(\Z)$-equivalence class $[Q]$, span the space of weight $2k$ cusp forms.  These $f_{k,\delta,[Q]}$ may also be viewed as (hyperbolic) Poincar\'e series which appeared in \cite{Pe3}; a good overview of these Poincar\'e series and their connection to $
f_{k,\delta,[Q]}
$ may be found in \cite{ImOs}. 

Choosing the discriminant to be negative instead, Bengoechea \cite{Bengoechea}, in her thesis, constructed meromorphic modular forms with poles at CM-points.  To describe these, for a positive-definite integral binary quadratic form $Q$ of discriminant $-D<0$, define
\begin{equation*}
f_Q(z)=f_{k,-D,[Q]}(z):=D^{\frac{k}{2}} \sum_{\mathcal{Q}\in [Q]}\mathcal{Q}(z,1)^{-k}.
\end{equation*}

The $f_Q$ are {\it meromorphic cusp forms}, i.e., meromorphic modular forms which vanish towards infinity like cusp forms. Given the connection between \eqref{Zagiercusp} and theta lifts, von Pippich and  the authors \cite{BKCycle} investigated the interrelation between theta lifts and the sum
$$
f_{k,-D}:=\sum_{Q\in \mathcal{Q}_{-D}/\SL_2(\Z)} f_{k,-D,[Q]}
$$
over all equivalence classes of (not necessarily positive-definite) $\mathcal{Q}_{-D}/\SL_2(\Z)$.  Here a regularized theta lift is required.  In Theorem \ref{thm:liftfkd}, we see that the singularities at the cusps are turned into singularities in the upper half-plane under this lift.  A related construction of certain 
non-meromorphic modular forms yield analogous results.  To state these, one defines weight $2-2k$ non-meromorphic modular forms  
 \begin{equation}\label{eqn:Gdef}
\mathcal{G}_{Q}(z)=\mathcal{G}_{1-k,-D,[Q]}(z):=D^{\frac{1-k}{2}}\sum_{\mathcal{Q}\in [Q]} \mathcal{Q}(z,1)^{k-1}\int_0^{\arctanh\left(\frac{\sqrt{D}}{\mathcal{Q}_{z}}\right)} \sinh^{2k-2}(\theta) d\theta,
\end{equation}
where for $\mathcal{Q}=[a,b,c]$ and $z=x+iy$
$$
\mathcal{Q}_{z}:=y^{-1}\left(a|z|^2+bx+c\right).
$$
There
 are analogous so-called locally harmonic Maass forms which occurred in work of H\"ovel \cite{Hoevel} for weight $0$ and in work of Kohnen and the authors \cite{BKW} for negative weight.

The functions $G_Q$ also show up as outputs of a theta lift (see Theorem \ref{thm:thetalift}) with inputs half-integral weight harmonic Maass forms with non-holomorphic principal parts at $i\infty$.  Furthermore, the associated theta kernel is closely related to that used to obtain $f_{k,-D}$,  and the analogous theta kernel for the omitted case $k=1$, which appeared in \cite{Hoevel}, was used by Alfes, Griffin, Ono, and Rolen \cite{AGOR} to construct a lift in the other direction, from integral to half-integral weight.

  The functions $\mathcal{G}_Q$ are also related to $f_Q$ via two differential operators from the theory of harmonic Maass forms (see Theorem \ref{thm:Gpolar}), which implies that they are what is known as polar harmonic Maass forms, the central objects in this article (see Section \ref{sec:polar} for the definition).  In addition,  the functions $\mathcal{G}_Q$ reappear when taking inner products of $f_Q$ against meromorphic cusp forms (see Theorem \ref{thm:innersimple}). Moreover, it turns out that the inner product of two of these meromorphic cusp forms agrees with the evaluation of higher Green's functions at CM-points (see Theorem \ref{thm:innerGreens}).  This also yields an identity between these inner products and heights of CM-cycles  (see Corollary \ref{cor:innerheight}).
	
We next turn to Fourier coefficients of meromorphic forms.  Before describing these, we first recall what is known for weakly holomorphic modular forms.  In work which gave birth to the Circle Method, Hardy and Ramanujan \cite{HR1, HR2} derived their famous asymptotic formula for the partition function $p(n)$ (see \eqref{eqn:partitiongrowth}).  Rademacher \cite{Rad} later perfected the method to derive an exact formula (see \eqref{Radformula}).  A key ingredient of the proof of \eqref{Radformula} is the fact that the partition generating function is essentially the reciprocal of a modular form with no roots in the upper half-plane, but which vanishes at the cusp $i\infty$ instead.  

Using modern techniques, a new formula for $p(n)$ as a (finite) trace of a certain weak Maass form evaluated at CM points of discriminant $1-24n$ modulo the action of $\Gamma_0(6)$ was recently proven by Bruinier and Ono \cite{BruinierOno}.  Much in the same way that the sum \eqref{Radformula} restricted to $k\ll \sqrt{n}$ gives a very good asymptotic approximation to $p(n)$, Masri \cite{Masri} used Bruinier and Ono's result to obtain a good asymptotic approximation to $p(n)$ with a shorter sum. The results above are not limited to the partition function and rather follow from a general structure for harmonic Maass forms.  Returning to Rademacher's formula \eqref{Radformula}, one sees that this is also part of a much more generic family of identities.  Rademacher and Zuckerman \cite{RZ, Zu1, Zu2} generalized \eqref{Radformula} to obtain exact formulas for the coefficients of all weakly holomorphic modular forms of negative weight.  Their formulas are explicit in the sense that one only requires the principal part of a given weakly holomorphic form.  
Conversely, one may ask whether one can detect modularity of a given Fourier expansion merely by showing that it has the same shape of Rademacher and Zuckerman's.  However, Knopp \cite{Knopp} determined that this was insufficient.  Viewed in a modern setting, Knopp found examples of such expansions which were only the holomorphic parts of harmonic Maass forms.  Moreover, using 
a basis \eqref{eqn:calFdef} of harmonic Maass forms with the simplest principal parts, one can show that the holomorphic parts of Fourier coefficients of all harmonic Maass forms have the same shape as Rademacher and Zuckerman's expansions \cite{BO2}.  The realization of the mock theta functions, mysterious functions introduced in Ramanujan's last letter to Hardy, as holomorphic parts of harmonic Maass forms by Zwegers \cite{Zwegers} also implies that the coefficients of these have the same shape.  This approach was used in \cite{BOMockTheta} to prove the Andrews--Dragonette conjecture \cite{Andrews,Dragonette} about one of Ramanujan's mock theta functions.

Until recently much less was known about Fourier coefficients of meromorphic cusp forms.  Hardy and Ramanujan \cite{HR3} considered the special case where the form has a unique simple pole modulo the action of $\SL_2(\Z)$.  In particular, they found a formula for the reciprocal of the weight $6$ Eisenstein series $E_6$.  Ramanujan \cite{RaLost} then stated further formulas for other meromorphic functions, such as the reciprocal for $E_4$ (see \eqref{eqn:1/E_4} for his formula), but, as usual for his writing, did not provide a proof. His claims concerning meromorphic cusp forms with simple poles were then subsequently proven by Bialek in his Ph.D. thesis written under Berndt \cite{Bi}.  Berndt, Bialek, and Yee \cite{BeBiYe} were first to explicitly compute the Fourier coefficients of meromorphic cusp forms with second-order poles, resolving the last of Ramanujan's claims about the coefficients of meromorphic cusp forms. The investigations in \cite{Bi,BeBiYe,HR3} all used the Circle Method, but the calculations become exceedingly more difficult as the order of the poles increase.  
More recently, Sebbar and Sebbar \cite{SebbarSebbar} related these meromorphic cusp forms with simple poles to equivariant functions and obtained the Fourier expansion in another form.  An alternate approach of Petersson \cite{Pe1} 
employs Poincar\'e series to yield formulas resembling those of Hardy and Ramanujan if the poles are all simple.  Combining this idea  with the construction of new Poincar\'e series and embedding the problem into the framework of polar harmonic Maass forms, we see in Theorem \ref{Fouriercoefficients} that the Fourier coefficients of all meromorphic cusp forms indeed have the same shape. For example, if the only pole in $\SL_2(\Z)\backslash\H$ is at an elliptic fixed point, then the coefficients may be written as a series over certain ideals (see \eqref{Fklr}).  This result is explicit in the sense that the Fourier expansion is directly given for a set of basis elements of a certain subspace of polar harmonic Maass forms.  
One may again obtain a more general formula for the holomorphic parts of polar harmonic Maass forms and the expansion can once more 
 be given by determining the principal part of the form.  We also note that some study has been done on congruences for meromorphic modular forms; for example, see the work of Honda and Kaneko \cite{HondaKaneko}.  It may be interesting to interpret and possibly generalize these results in the setting of polar harmonic Maass forms.

One may wonder why polar harmonic Maass forms are employed if the ultimate goal is the Fourier coefficients of meromorphic modular forms.  One of the many reasons is that polar harmonic Maass forms are simpler than meromorphic modular forms to construct because  forms with arbitrary principal parts exist.   If the principal parts satisfy a certain condition dictated by the Riemann--Roch Theorem, then the polar harmonic Maass form is necessarily a meromorphic modular form.  In the reverse direction, one may also obtain an explicit construction of all meromorphic modular forms by determining when a given polar harmonic Maass form is meromorphic.  By investigating the Fourier coefficients of the basis elements, we obtain the Fourier coefficients of every meromorphic modular form in particular.  
More specifically, given only the principal parts of a form one may write it as a linear combination of Poincar\'e series built by Fay \cite{Fay} 
with easily-determined principal parts.
Moreover, these 
are preimages of the elliptic Poincar\'e series under natural differential operators (see Theorem \ref{thm:Poincmain}).  To obtain the coefficients in the style of Hardy and Ramanujan's expansion for $1/E_6$, 
an alternate representation of these Poincar\'e series 
is given in Lemma \ref{lem:fsplit}.  

The paper is organized as follows.  In Section \ref{sec:polar}, we introduce polar harmonic Maass forms and their Fourier and elliptic expansions as well as the connection between these expansions and a pairing between weight $2-2k$ polar harmonic Maass forms and weight $2k$ cusp forms.  In Section \ref{sec:Poincare}, we recall some properties of known Poincar\'e series.  In Section \ref{sec:PeterssonSplit}, a splitting of elliptic Poincar\'e series is explained by placing them into the framework of polar harmonic Maass forms, and certain bases are also constructed.  We give the Fourier expansions of polar harmonic Maass forms, and therefore also meromorphic modular forms, in Section \ref{sec:Fourier}.  The connections between $f_Q$ and $\mathcal{G}_Q$ are investigated in Section \ref{sec:cycle}.  We finish the paper with some discussion about possible future directions in Section \ref{sec:future}.

\section{Polar harmonic Maass forms}\label{sec:polar}

\subsection{Basic definitions}\label{2.1}
For $M=\left(\begin{smallmatrix}a &b\\c &d\end{smallmatrix}\right)\in\SL_2(\Z)$ and a function $f\colon \H\to\C$, we require the weight $\wt\in\frac{1}{2}\Z$ slash action
$$
f|_{\wt}M(z):= \begin{cases}
                        (cz +d)^{-\kappa} f(M z) & \text{ if } \kappa \in \Z,\\
                        \left(\frac{c}{d}\right)^{-2\kappa} \varepsilon_d^{2\kappa}(cz +d)^{-\kappa} f(M z) & \text{ if } \kappa \in \frac{1}{2}\Z \backslash \Z.
                       \end{cases}
$$
Here $(\frac{\cdot}{\cdot})$ denotes the Kronecker symbol and 
$$
\varepsilon_d:=\begin{cases} 1 &\text{if }d\equiv 1\pmod{4},\\ i&\text{if }d\equiv 3\pmod{4}.\end{cases}
$$
\begin{definition}
For $N\in\N$ and $
1\neq \kappa \in\frac{1}{2}\Z
$, a \begin{it}polar harmonic Maass form\end{it} of weight $\kappa$ on $\Gamma_0(N)$ is a function $\mathcal{F}:\H\to\C$ which is real analytic outside of a discrete set of points and satisfies the following conditions: 
\noindent

\noindent
\begin{enumerate}[leftmargin=*]
\item
For every $M\in\Gamma_0(N)$, we have $\mathcal{F}|_{\kappa}M=\mathcal{F}$. 
\item 
The function $\mathcal{F}$ is annihilated by the \begin{it}weight $\kappa$ hyperbolic Laplacian\end{it} 
$$
\Delta_{\kappa}:=-y^2\left(\frac{\partial^2}{\partial x^2}+\frac{\partial^2}{\partial y^2}\right)+i\kappa y\left(\frac{\partial}{\partial x}+i\frac{\partial}{\partial y}\right).
$$
\item For every $\mathfrak{z}\in \H$, there exists $n\in\N_0$ such that $(z-\mathfrak{z})^n\mathcal{F}(z)$ is bounded in some neighborhood of $\mathfrak{z}$.
\item The function $\mathcal{F}$ grows at most linear exponentially towards cusps of $\Gamma_0(N)$.
\end{enumerate}
If one allows a general eigenvalue in (2), then one obtains polar Maass forms.  If $\mathcal{F}$ does not have any singularities in $\H$, i.e., if $n=0$ in condition (3) for every $\mathfrak{z}\in\H$, then $\mathcal{F}$ is a \begin{it}harmonic Maass form\end{it}; such functions were first introduced in Section 3 of \cite{BF}.  
Using a slight variant of the operator $\Delta_{\kappa}$, Fay \cite{Fay} also studied Maass forms, although not specializing on the harmonic case.

\end{definition}

We denote the space of polar harmonic Maass forms by $\mathcal{H}_{\kappa}(N)$ and those whose singularities in $\H$ are all poles and which are bounded towards all cusps by $\H_{\kappa}(N)$.  
We furthermore let $H_{\kappa}(N)$ 
be the space of harmonic Maass forms and let $M_{\kappa}^!(N)$ stand for  the subspace of \begin{it}weakly holomorphic modular forms\end{it}, 
 those meromorphic modular forms whose poles (if any) are supported at cusps.  The subspace of meromorphic modular forms we further denote by  $\mathbb{M}_{\kappa}(N)$ and use the notation $\SSnew_{\kappa}(N)$ for the subspace of {\it meromorphic cusp forms}, i.e., those meromorphic modular forms $f$ for which $y^{\frac{\kappa}{2}}|f(z)|$ decays towards all cusps.  As usual, the space of cusp forms is $S_{\kappa}(N)$.  Throughout, we omit $N$ if $N=1$ and $\kappa\in\Z$ or $N=4$ and $\kappa\in\frac{1}{2}\Z\setminus\Z$.  

\par
An important subspace of $\mathcal{H}_{\kappa}(N)$ is obtained by noting that the hyperbolic Laplacian splits as
\begin{equation}\label{eqn:Deltasplit}
\Delta_\kappa=-\xi_{2-\kappa}\circ \xi_\kappa,
\end{equation}
where $\xi_{\kappa}:=2iy^{\kappa} \overline{\frac{\partial}{\partial \overline{z}}}$.  If $\mathcal{F}$ satisfies weight $\kappa$ modularity, then $\xi_{\kappa}(\mathcal{F})$ is modular of weight $2-\kappa$ and one sees from the decomposition \eqref{eqn:Deltasplit} that $\xi_{\kappa}(\mathcal{F})\in \mathbb{M}_{2-\kappa}(N)$ if $\mathcal{F}\in \mathcal{H}_{\kappa}(N)$.  It is thus natural to consider the subspace $\mathcal{H}_{\kappa}^{\operatorname{cusp}}(N)\subseteq \mathcal{H}_{\kappa}(N)$ consisting of those $\mathcal{F}$ for which $\xi_{\kappa}(\mathcal{F})$ is a cusp form.  
We similarly apply the superscript ``cusp'' to denote the intersection of other subspaces with $\mathcal{H}_{\kappa}^{\operatorname{cusp}}(N)$, such as $H_{\kappa}^{\operatorname{cusp}}(N):=H_{\kappa}(N)\cap \mathcal{H}_{\kappa}^{\operatorname{cusp}}(N)$.  
If $g:=\xi_{\kappa}(\mathcal{F})\in S_{\kappa}(N)$, then there is another non-modular natural preimage of $g$, namely the \begin{it}non-holomorphic Eichler integral\end{it} (cf. \cite{ZagierBourbaki})
\begin{equation}\label{eqn:g^*def}
g^*(z):=(2i)^{1-\kappa}\int_{-\overline{z}}^{i\infty} g^{\operatorname{c}}(w) (w+z)^{1-\kappa} dw,
\end{equation}
where $g^c (w) := \overline{g(-\overline{w})}$. Then 
\begin{equation}\label{eqn:xig*}
\xi_{\kappa}(g^*)=g. 
\end{equation}
In addition to $\xi_{\kappa}$, if $\kappa\in-2\N_0$ there is another natural differential operator $\mathcal{D}^{1-\kappa}$ from $\mathcal{H}_{\kappa}(N)$ to $\mathbb{M}_{2-\kappa}(N)$, where $\mathcal{D}:=\frac{1}{2\pi i}\frac{\partial}{\partial z}$. 
An identity of Bol relates this operator to 
the \begin{it}raising operator\end{it} $R_{\kappa,z}:=2i\frac{\partial}{\partial z} + \kappa y^{-1}$.  
In particular, denoting by $R_{\kappa,z}^n:=R_{\kappa+2n-2,z}\circ\dots\circ R_{\kappa,z}$ repeated raising, $\mathcal{D}^{1-\kappa}$ is a constant multiple of $R_{\kappa}^{1-\kappa}$.

\subsection{Fourier and elliptic expansions}
Polar harmonic Maass have a natural decomposition into holomorphic and non-holomorphic parts. Both parts can contain singularities; the singularities in the holomorphic part are poles, while one can determine the kind of singularities in the non-holomorphic part by noting that its image under $\xi_{\kappa}$ is meromorphic.  To describe the principal parts of polar harmonic Maass forms, we require their Fourier expansion around $\varrho\in \mathcal{S}_N$, the set of inequivalent cusps of $\Gamma_0(N)$.  Suppose that the cusp width  of $\varrho$ is $\ell_{\varrho}$ and choose $M_{\varrho}$ such that $M_{\varrho}\varrho=i\infty$.  The Fourier expansion at $\varrho$ then has the shape (convergent for $y \gg 0$) 
$$
\mathcal{F}_{\varrho}(z):=\mathcal{F}\big|_{\kappa} M_{\varrho}(z)=\mathcal{F}_{\varrho}^+(z) + \mathcal{F}_{\varrho}^-(z),
$$
where, for some $c_{\mathcal{F},\varrho}^{\pm}(n)\in\C$, 
\begin{align}\nonumber
\mathcal{F}_{\varrho}^+(z)&:=\sum_{n\gg -\infty} c_{\mathcal{F},\varrho}^+(n) e^{\frac{2\pi i nz}{\ell_{\varrho}}},\\
\label{eqn:F-exp}
\mathcal{F}_{\varrho}^-(z)&:=c_{\mathcal{F},\varrho}^-(0) y^{1-\kappa}+\sum_{\substack{n\ll \infty\\ n\neq 0}} c_{\mathcal{F},\varrho}^-(n) \Gamma\left(1-\kappa,-\frac{4\pi n y}{\ell_{\varrho}}\right)e^{\frac{2\pi i nz}{\ell_{\varrho}}},
\end{align}
with the incomplete gamma function $\Gamma(j,v):=\int_{v}^{\infty} t^{j-1} e^{-t} dt$.
We call $\mathcal{F}_{\varrho}^+$ the
\begin{it}holomorphic part\end{it} of $F$ at $\varrho$ 
and $\mathcal{F}_{\varrho}^-$ the \begin{it}non-holomorphic part\end{it}. We may omit the dependence on $\mathcal{F}$ and $\varrho$ if it is clear from the context, 
while if $\mathcal{F}\in\mathbb{M}_{\kappa}(N)$, then we simply write $c_{\mathcal{F},\varrho}(n)$ instead of $c_{\mathcal{F},\varrho}^+(n)$. For $\mathcal{F}\in 
\mathcal{H}_{\kappa}(N)$ and $\varrho\in \mathcal{S}_N$, we call the terms of the Fourier expansion which grow towards $\varrho$ the \begin{it}principal part\end{it} (at $\varrho$).

 The coefficients $c_{\mathcal{F},\varrho}^-(n)$ are closely related to coefficients of meromorphic modular forms of weight $2-\kappa$.  Indeed, this relationship follows from the fact that if $\mathcal{F}$ is modular of weight $\kappa$, then $\xi_{\kappa}(\mathcal{F})$ is modular of weight $2-\kappa$, and thus $\xi_{\kappa}$ maps weight $\kappa$ polar
harmonic Maass forms to weight $2-\kappa$ meromorphic modular forms.

We next consider elliptic expansions of polar harmonic Maass forms.  For this, let
$r_{\mathfrak{z}}(z):=|X_{\mathfrak{z}}(z)|$ with $X_{\mathfrak{z}}(z):= \frac{z-\mathfrak{z}}{z-\overline{\mathfrak{z}}}$ and for $0\leq {v}<1$ and $\kappa\in -\N_0$ 
define
\begin{equation*}
\beta_0\left({v};a,b\right):=\beta\left({v};a,b\right)-\mathcal{C}_{a,b}
\end{equation*}
where  $\beta({v};a,b):=\int_0^{v} t^{a-1} (1-t)^{b-1} dt$ is the \begin{it}incomplete beta function\end{it} and 
\begin{equation*}
\mathcal{C}_{a,b}:=\sum_{\substack{0\leq j\leq a-1\\ j\neq -b}} \binom{a-1}{j}\frac{(-1)^{j+1}}{j+b}.
\end{equation*}
Here, for $b\in\N$, we have
$$
\beta_0(v;a,b)=-\beta(1-v;b,a)=-\frac{(1-v)^b}{b}{_2F_1}(b,,1-a;1+b;1-v).
$$
\rm

The elliptic expansion of a polar harmonic Maass form is given in Proposition 2.2 of \cite{BKWeight0}.  However, it was later pointed out by Pioline that the following proposition is a special case of Theorem 1.1 of \cite{Fay}, where a more general expansion is given for  polar Maass forms, 
although in another guise.

\begin{proposition}\label{prop:ellexp}
Suppose that $\kappa\in -2\N_0$ and $\mathfrak{z}\in\mathbb{H}$.  
\noindent

\noindent
\begin{enumerate}[leftmargin=*]
\item[\rm(1)] Assume that $\mathcal{F}$ satisfies $\Delta_{\kappa} (\mathcal{F}) = 0$ and that for some $n_0\in \N$ 
the function $r_{\mathfrak{z}}^{n_0}(z)
 \mathcal{F}(z)$ is bounded in some neighborhood $\mathcal{N}$ around $\mathfrak{z}$. Then there exist $a_{\mathcal{F},\mathfrak{z}}^{\pm}(n) \in\C$, such that for $z\in\mathcal{N}$ 
\begin{equation}\label{elliptic}
\quad\mathcal{F}(z)=\left(z-\overline{\mathfrak{z}}\right)^{-\kappa}\left(\sum_{n\gg -\infty} a_{\mathcal{F},\mathfrak{z}}^+(n) X_{\mathfrak{z}}^n(z) + \sum_{n\ll\infty } a_{\mathcal{F},\mathfrak{z}}^-(n)\beta_0\left(1-r_{\mathfrak{z}}^2(z);1-\kappa,-n\right) X_{\mathfrak{z}}^n(z)\right).
\end{equation}
\item[\rm(2)]
If $\mathcal{F}\in \mathcal{H}_{\kappa}(N)$, then \eqref{elliptic} runs only over those $n$ which satisfy $n\equiv -\kappa/2\pmod{\omega_{\mathfrak{z},N}}$, where $\omega_{\mathfrak{z},N}:=\#\Gamma_{\mathfrak{z},N}$.  Here $\Gamma_{\mathfrak{z},N}:= \Gamma_{\mathfrak{z}}\cap 
\Gamma_0(N)
$ with $\Gamma_{\mathfrak{z}}$ 
the stabilizer of $\mathfrak{z}$ in $\operatorname{PSL}_2(\Z)$.  Furthermore, if $\mathcal{F}\in \mathcal{H}_{\kappa}^{\operatorname{cusp}}(N)$, then the second sum only runs over $n<0$. 
\end{enumerate}
\end{proposition}
\begin{remark}
The statement in Proposition 2.2 of \cite{BKWeight0} is slightly different; the functions $\beta_0$ are replaced with incomplete beta functions $\beta$.  These functions only differ by a constant, so the change in this version only differs by changing the coefficients of the meromorphic part.
\end{remark}
We define the \begin{it}meromorphic part of the ellipic expansion\end{it} around $\mathfrak{z}$ by 
$$
\mathcal{F}_{\mathfrak{z}}^+ (z):=\left(z-\overline{\mathfrak{z}}\right)^{-\kappa}\sum_{n\gg -\infty}a_{\mathcal{F},\mathfrak{z}}^+(n)X_{\mathfrak{z}}^n(z)
$$
and the \begin{it}non-meromorphic part of the elliptic expansion\end{it} by 
$$
\mathcal{F}_{\mathfrak{z}}^- (z) :=\left(z-\overline{\mathfrak{z}}\right)^{-\kappa}\sum_{n\ll \infty} a_{\mathcal{F},\mathfrak{z}}^-(n)\beta_0\left(1-r_{\mathfrak{z}}^2(z);1-\kappa,-n\right) X_{\mathfrak{z}}^n(z).
$$

\subsection{Bruinier--Funke pairing}

For $g\in S_{2k}(N)$  and  $\mathcal{F}\in \mathcal{H}_{2-2k}^{\operatorname{cusp}}(N),$ following Bruinier and Funke \cite{BF}, define the \textit{pairing}
\begin{equation*}
\{g, \mathcal{F}\}:=\left<g, \xi_{2-2k}(\mathcal{F})\right>,
\end{equation*}
where, for $g,h\in S_{2k}(N)$ and $\mu_N:=\left[\SL_2(\Z):\Gamma_0(N)\right]$,
\begin{equation}\label{eqn:innerclassical}
\left<g,h\right>:=\frac{1}{\mu_N}\int_{\Gamma_0(N)\backslash \H} g(z)\overline{h(z)} y^{2k} \frac{dx dy}{y^2}
\end{equation}
is the standard \begin{it}Petersson inner-product\end{it}.  The pairing $\{g,\mathcal{F}\}$ was computed for $\mathcal{F}\in H_{2-2k}^{\operatorname{cusp}}(N)$ by Bruinier and Funke in Proposition 3.5 of \cite{BF}.  We recall an extension of their evaluation of $\{g,\mathcal{F}\}$ to the entire space $\mathcal{H}_{2-2k}^{\operatorname{cusp}}(N)$, letting $\mathfrak{z}=\mathfrak{z}_1+i\mathfrak{z}_2$ throughout.
\begin{proposition}[Proposition 6.1 of \cite{BKWeight0}]\label{expansionprod}
If $g\in S_{2k}(N)$ and $\mathcal{F}\in\mathcal{H}_{2-2k}^{\operatorname{cusp}}(N)$, then
\begin{equation}\label{eqn:gFpair}
\{g, \mathcal{F}\}=\frac{\pi}{\mu_N}\sum_{\mathfrak{z}\in\Gamma_0(N)\backslash\H} \frac{1}{\mathfrak{z}_2 \omega_{\mathfrak{z},N}}\sum_{n\geq 1}a_{\mathcal{F},\mathfrak{z}}^+\left(-n\right) a_{g,\mathfrak{z}}\left(n-1\right) + \frac{1}{\mu_N}\sum_{\varrho\in\mathcal{S}_N} \sum_{n\geq 1} c_{\mathcal{F},\varrho}^{+}(-n)c_{g,\varrho}(n).
\end{equation}
\end{proposition}
\begin{proof}[Sketch of proof]
One uses Stokes Theorem for a fundamental domain with small neighborhoods cut out around each of the cusps and the singularities of $\mathcal{F}$.  Integrating along the boundary of these neighborhoods yields a product of the Fourier coefficients from the expansions around the cusps and a product of the elliptic coefficients near each pole.
\end{proof}
\subsection{Green's functions}\label{sec:Greens}
Recall that for $k\in\mathbb{N}_{>1}$ and $\Gamma\subset \SL_2(\mathbb{Z})$ of finite index, the \textit{higher Green's function} $G^{\mathbb{H}/\Gamma}_k: \mathbb{H}\times\mathbb{H}\to\mathbb{C}$ is uniquely characterized by the following properties:
\noindent

\noindent
\begin{enumerate}[leftmargin=*]
\item $G^{\mathbb{H}/\Gamma}_k$ is a smooth real-valued function on $\mathbb{H}\times\mathbb{H}\setminus \{(z, \gamma z)|\gamma\in\Gamma, z\in\mathbb{H}\}.$
\item For $\gamma_1, \gamma_2\in\Gamma$, we have $G^{\mathbb{H}/\Gamma}_k(\gamma_1 z, \gamma_2\mathfrak{z})= G^{\mathbb{H}/\Gamma}_k(z, \mathfrak{z}).$
\item We have 
\[
\Delta_{0, z}\left(G^{\mathbb{H}/\Gamma}_k\left(z, \mathfrak{z}\right)\right)
=\Delta_{0, \mathfrak{z}}
\left(G^{\mathbb{H}/\Gamma}_k\left(z, \mathfrak{z}\right)\right)=k(1-k)G^{\mathbb{H}/\Gamma}_k\left(z, \mathfrak{z}\right).
\]
\item As $z\to\mathfrak{z}$
\[
G^{\mathbb{H}/\Gamma}_k(z, \mathfrak{z})=2\omega_{\mathfrak{z}}\log\left(r_{\mathfrak{z}}(z)\right)+O(1).
\]
\item As $z$ approaches a cusp, $G^{\mathbb{H}/\Gamma}_k(z, \mathfrak{z})\to 0$.
\end{enumerate}

These higher Green's functions have a long history, appearing as special cases of the resolvent kernel studied by Fay \cite{Fay} and investigated thoroughly by Hejhal in \cite{Hejhal}, for example.  Gross and Zagier \cite{GZ} conjectured
\rm
that their evaluations at CM-points are essentially logarithms of algebraic numbers. In the special case that 
the space of weight $2k$ cusp forms on $\Gamma$ is trivial,  the conjecture reads as
\[
G^{\mathbb{H}/\Gamma}_k(z, \mathfrak{z})=(D_1 D_2)^{\frac{1-k}{2}}\log(\alpha)
\]
for CM-points $z, \mathfrak{z}$  of discriminants $D_1 \text{ and } D_2$, respectively and $\alpha$ some algebraic number. Various cases of this conjecture have been solved. For example, Mellit, in his Ph.D. thesis \cite{Mellit}, proved the case $k=2, \mathfrak{z}=i$ and also gave an interpretation to $\alpha$ as a certain intersection number of certain higher Chow cycles.

\section{Poincar\'e series}\label{sec:Poincare}
\subsection{Maass Poincar\'e series}

An important tool to construct automorphic forms are \begin{it}Poincar\'e series\end{it}, which have a long history going back to Poincar\'e \cite{PoincareSeries}.  To be more precise, for $\kappa>2$ and $m\in\Z$, the classical (weakly) holomorphic Poincar\'e series are defined by (see (2b.1) of \cite{Pe1})
\begin{equation}\label{eqn:weakPoinc}
\mathcal{P}_{\kappa,m,N}(z):={\displaystyle{\sum_{M\in \Gamma_{\infty}\backslash\Gamma_0(N)}}} e^{2\pi imz}\bigg|_{\kappa}M \in M_{\kappa}^!(N),
\end{equation}
where $\Gamma_{\infty}:=\left\{\pm \left(\begin{smallmatrix}1&n\\ 0 &1\end{smallmatrix}\right)\middle|n\in\Z\right\}$.  We further set
$$
P_{\kappa,m,N}:=\mathcal{P}_{\kappa,m,N}\Big|\operatorname{pr}\in M_{\kappa}^!(N),
$$
where $|\operatorname{pr}$ is the identity for $\kappa\in\Z$ and the projection operator (see p. 250 of \cite{KohnenFourier} for a definition) into Kohnen's plus space if $\kappa\notin\Z$.  Here the restriction $\kappa>2$ is made so that the series converge absolutely and uniformly on compact sets.  

To construct harmonic Maass forms, for $\kappa<0$ and $m\in \Z\setminus\{0\}$, we define the harmonic function
$$
\varphi_{\kappa,m}(z):=\frac{(-\sgn(m))^{1-\kappa}}{\Gamma(2-\kappa)}(4\pi |m|y)^{-\frac{\kappa}{2}}M_{\sgn(m)\frac{\kappa}{2},\frac{1-\kappa}{2}}(4\pi |m|y)e^{2\pi i mx},
$$
where $M_{\mu, \nu}$ is the $M$-Whittaker function.  
These forms appeared in applications for a number of authors (cf. \cite{BF,Fay}) following 
work of Niebur \cite{Niebur}, who built eigenfunctions under $\Delta_0$. For $\kappa\in-\frac{1}{2}\N$ and $N\in\N$, we then set 
$$
\mathcal{F}_{\kappa,m,N}:=\sum_{M\in\Gamma_{\infty}\backslash\Gamma_0(N)} \varphi_{\kappa,m}\bigg|_{\kappa}M\in H_{\kappa}(N).
$$
and 
\begin{equation}\label{eqn:calFdef}
F_{\kappa,m,N}:=\mathcal{F}_{\kappa,m,N}\Big|\operatorname{pr}\in H_{
\kappa
}(N).
\end{equation}
For $m<0$, the functions $F_{\kappa,m,N}$ have principal parts $q^{m}$ at $i\infty$ whereas for $m>0$ they have a non-holomorphic principal part.  Furthermore, a straightforward calculation, using the fact that $\xi$ commutes with the slash action, yields for $m\in \Z\setminus\{0\}$ (e.g., see the displayed formula after (6.8) of \cite{BOR}) 
\begin{equation*}
\xi_{\kappa}\left(F_{\kappa,m,N}\right) = \frac{(-1)^{\kappa-1}(4\pi m)^{1-\kappa}}{
\Gamma(1-\kappa)
}P_{2-\kappa,-m,N}.
\end{equation*}
If $\kappa\in-\N$, then $F_{\kappa,m,N}$ also has a nice image under $\mathcal{D}^{1-\kappa}$.  Specifically, Bruinier, Ono, and Rhoades showed in (6.8) of \cite{BOR} (see also Satz 9 of \cite{Pe1}, where part of the computation was carried out) that
$$
\mathcal{D}^{1-\kappa}\left(\mathcal{F}_{\kappa,m,N}\right)=m^{1-\kappa} P_{2-\kappa,m,N}.
$$

\subsection{Elliptic Poincar\'e series}

For $N\in\N$, $n\in\Z$, and $2<\kappa\in \frac{1}{2}\Z$, we next recall the elliptic Poincar\'e series, defined in  \cite{Pe3,Pe1},
\begin{equation}\label{eqn:Psidef}
\Psi_{\kappa,n,N}(z, \mathfrak{z}):=\sum_{M\in\Gamma_0(N)}\frac{(z-\mathfrak{z})^n}{(z -\overline{\mathfrak{z}})^{\kappa+n}}\bigg|_{\kappa,z}M.
\end{equation}
The functions $z\mapsto \Psi_{\kappa,n,N}(z,\mathfrak{z})$ are weight $\kappa$ meromorphic modular forms which are cusp forms if $n\in\N_0$ and have a pole of order $|n|$ if $n\in -\N_0$.  We also note that 
\begin{equation}\label{eqn:Psi^c}
\Psi_{\kappa,n,N}^{\operatorname{c}}(z,\mathfrak{z})=\Psi_{\kappa,n,N} (-z,\overline{\mathfrak{z}}).
\end{equation}

Moreover, for $n=-1$, these are closely related to other Poincar\'e series defined 
in (2b.9) of \cite{Pe1} by 
\begin{equation*}
H_{\kappa,N}(z,\mathfrak{z}):=2\sum_{M\in \Gamma_{\infty}\backslash \Gamma_0(N)} \frac{1}{1-e^{2\pi i (\mathfrak{z}-z)}}\bigg|_{\kappa,z} M.
\end{equation*}
To state the connection, define
 $$
K_{\kappa,N}(z,\mathfrak{z}):=\frac{i}{\pi}(2i\mathfrak{z}_2)^{\kappa-1}\sum_{n\in \N} \mathcal{I}_{\kappa,-n}(\mathfrak{z}_2)e^{-2\pi in\mathfrak{z}}\mathcal{P}_{\kappa,n,N}(z)
$$
with (for some $-2\mathfrak{z}_2<\alpha<0$ 
and $r\in\Z$)
\begin{equation*}
\mathcal{I}_{\kappa,r}(\mathfrak{z}_2):=\int_{i\alpha -\infty}^{i\alpha +\infty} \frac{e^{2\pi i r t}}{\left(t+2i\mathfrak{z}_2\right)^{\kappa-1}} \frac{dt}{t}.
\end{equation*}
Petersson showed in (5b.6) of \cite{Pe1} that
\begin{equation}\label{eqn:split}
\left(2i\mathfrak{z}_2\right)^{\kappa-1}\Psi_{\kappa,-1,N}(z,\mathfrak{z})=2\pi i H_{\kappa,N}(z,\mathfrak{z}) -2\pi i K_{\kappa,N}(z,\mathfrak{z}).
\end{equation}
While he recognized the left-hand side of \eqref{eqn:split} as a non-holomorphic modular form of weight $2-\kappa$, as a function of $\mathfrak{z}$, and investigated the meromorphic functions $H_{\kappa,N}$, on page 67 of \cite{Pe1} he questioned the place of the functions $K_{\kappa,N}$ within the framework of automorphic forms.  They are cusp forms as a function of $z$, but their properties as a function of $\mathfrak{z}$ were rather mysterious. 
This was investigated in \cite{BKFourier}, and its connection to 
the elliptical Poincar\'e series, defined in \eqref{eqn:Psidef}, is given in Theorem \ref{ConjectureTheorem}.

\section{Understanding Petersson's splitting and bases for polar harmonic Maass forms}\label{sec:PeterssonSplit}

\subsection{Petersson's splitting}
Throughout this section, $2<\kappa \in \frac{1}{2}\mathbb{N}$.  One can recognize \eqref{eqn:split} as 
the usual 
 splitting of a weight $2-\kappa$ polar harmonic Maass form into its holomorphic and non-holomorphic parts around $i\infty$.
\begin{proposition}[Proposition 3.1 of \cite{BKFourier}]\label{prop:HkPsi}
For $N\in\N$, $\mathfrak{z}\mapsto\mathfrak{z}_2^{\kappa-1}\Psi_{\kappa,-1,N}(z,\mathfrak{z})\in \H_{2-\kappa}^{\operatorname{cusp}}(N)$ with meromorphic part
$$
\frac{\pi }{(2i)^{\kappa-2}}H_{\kappa,N}(z,\mathfrak{z}).
$$
\end{proposition}
Proposition \ref{prop:HkPsi} implies that the functions $K_{\kappa,N}$ are the non-holomorphic parts of polar harmonic Maass forms.  In particular, they are non-holomorphic Eichler integrals of the functions $\Psi_{\kappa,0,N}(\mathfrak{z},z)$.
\begin{theorem}[Theorem 1.1 of \cite{BKFourier}]\label{ConjectureTheorem}
We have 
\begin{equation}\label{eqn:conj}
K_{\kappa,N}\left(z,\mathfrak{z}\right)=-\frac{i^{\kappa}(\kappa-1)}{2^{\kappa-1}\pi}\Psi_{\kappa,0,N}^{*}(\mathfrak{z},z).
\end{equation}
\end{theorem}
\begin{remark}
The result in \cite{BKFourier} is much more general, considering real weight with multiplier on more general groups.  In this note, we only introduce polar harmonic Maass forms in this restricted setting to avoid extra notation. The right-hand side of \eqref{eqn:conj} was also explicitly given as an integral in \cite{BKFourier}, which follows by \eqref{eqn:g^*def} and \eqref{eqn:Psi^c}.

\end{remark}
\begin{proof}[Sketch of proof]
By Proposition \ref{prop:HkPsi}, $K_{\kappa,N}$ is (up to a constant) the non-holomorphic part of $\mathfrak{z}_2\mapsto\mathfrak{z}_2^{\kappa-1}\Psi_{\kappa,-1,N}(z,\mathfrak{z})$. A short computation shows that 
\begin{equation*}
\xi_{2-\kappa,\mathfrak{z}}\left(\mathfrak{z}_2^{\kappa-1}\Psi_{\kappa,-1,N}(z,\mathfrak{z})\right)= (\kappa-1)e^{\pi i\kappa}\Psi_{\kappa,0,N}(\mathfrak{z},z).
\end{equation*}    
Moreover, \eqref{eqn:xig*} implies that the right-hand side of \eqref{eqn:conj} is also mapped to a constant multiple of $\Psi_{\kappa,0,N}(\mathfrak{z},z)$.  Hence the difference is annihilated by $\xi_{2-\kappa}$, implying that it is meromorphic.  However, one can show that both sides of \eqref{eqn:conj} have expansions of the type \eqref{eqn:F-exp}, and such an expansion can only be meromorphic if all coefficients 
are zero.  Hence the difference 
vanishes, giving the claim.

\end{proof}

\subsection{Bases for meromorphic modular forms and polar harmonic Maass forms}\label{sec:basis}
One obtains a basis for the space $\H_{2-\kappa,N}^{\operatorname{cusp}}$ by repeatedly applying 
repeated raising to $\mathfrak{z}_2^{\kappa-1}\Psi_{\kappa,-1,N}(z,\mathfrak{z})$.  Note that the non-holomorphic part of a polar harmonic Maass form vanishes if and only if the form is a meromorphic modular form.  Hence comparing meromorphic parts and using Proposition \ref{prop:HkPsi} gives the following lemma.
\begin{lemma}[Proposition 4.2 of \cite{BKFourier}]\label{lem:fsplit}
If $f\in \SSnew_{2-\kappa}(N)$, then there exist $z_1,\dots,z_r\in \Gamma_0(N)\backslash\H$ and $a_{\ell,n}\in\C$ such that 
$$
f(\mathfrak{z}) = \sum_{\ell=1}^{r}\sum_{n=0}^{n_\ell} a_{\ell,n} R_{\kappa,z}^{n}\left[H_{\kappa,N}\left(z,\mathfrak{z}\right)\right]_{z=z_{\ell}}.
$$
\end{lemma}

Pioline pointed out an alternative way to obtain a basis of polar harmonic Maass forms which are bounded towards the cusps, namely using Fay's \cite{Fay} polar harmonic Poincar\'e series in negative weight.  The Poincar\'e series $\Psi_{\kappa,n,N}$ are formed by slashing the $n$th coefficient of the meromorphic part of the elliptic expansion \eqref{elliptic}.  Using the non-meromorphic analogue 
\begin{equation*}
\psi_{2-2k,n}(z,\mathfrak{z}):=\left(z-\overline{\mathfrak{z}}\right)^{2k-2} \beta \left( 1-r_{\mathfrak{z}}^2(z) ;2k-1, -n\right)X_{\mathfrak{z}}^n(z),
\end{equation*}
one may define the weight $2-2k<0$ Poincar\'e series
$$
\mathbb{P}_{2-2k,n,N}( z,\mathfrak{z}) := \sum_{M\in \Gamma_0(N)}\psi_{2-2k,n}(z,\mathfrak{z})\Big|_{2-2k,z} M.
$$
These Poincar\'e series appear as a special case of (44) in \cite{Fay} and they are connected to Petersson's Poincar\'e series.  The following essentially follows by Theorem 2.1 of \cite{Fay} (the notation for these Poincar\'e series is different in \cite{BJK}, because $\mathfrak{z}$ was only considered to be a parameter there)
\begin{theorem}[Theorem 4.3 of \cite{BJK}, for $N=1$]\label{thm:Poincmain}
For $k>1$, the Poincar\'e series $\mathbb{P}_{2-2k,n,N}(z,\mathfrak{z})$ converge absolutely and locally uniformly.  The functions $z\mapsto \mathbb{P}_{2-2k,n,N}(z,\mathfrak{z})$ span the space of weight $2-2k$ polar harmonic Maass forms on $\Gamma_0(N)$ which are bounded in all cusps.  Furthermore, 
\begin{align}\label{eqn:xiPoincn}
\xi_{2-2k,z}\left(\mathbb{P}_{2-2k,n,N}(z,\mathfrak{z})\right)&=(4\mathfrak{z}_2)^{2k-1}\Psi_{2k,-n-1,N}(z,\mathfrak{z}),\\
\label{eqn:DPoinc} \mathcal{D}_z^{2k-1}\left(\mathbb{P}_{2-2k,n,N}(z,\mathfrak{z})\right)&=-(2k-2)!\left(\frac{\mathfrak{z}_2}{\pi}\right)^{2k-1}\Psi_{2k,n+1-2k,N}(z,\mathfrak{z}).
\end{align}
The principal parts of $\mathbb{P}_{2-2k,n,N}$ are given by
$$
\begin{cases}
2\omega_{\mathfrak{z},N}\left(z-\overline{\mathfrak{z}}\right)^{2k-2}\beta_0\left(1-r_{\mathfrak{z}}^{2}(z);2k-1,-n\right)X_{\mathfrak{z}}^{n}(z)&\text{if }n\geq 0\text{ and }n\equiv k-1 \pmod {\omega_{\mathfrak{z},N}},\\
2\omega_{\mathfrak{z},N}\mathcal{C}_{2k-1,-n} \left(z-\overline{\mathfrak{z}}\right)^{2k-2}X_{\mathfrak{z}}^{n}(z)&\text{if }n<0\text{ and }n\equiv k-1 \pmod {\omega_{\mathfrak{z},N}},\\
0&\text{otherwise.}
\end{cases}
$$
Moreover, the functions $\mathfrak{z}\mapsto \mathfrak{z}_2^{2-2k+n} \mathbb{P}_{2-2k,n,N}(z,\mathfrak{z})$ satisfy weight $2k-2-2n$ modularity and are eigenfunctions under $\Delta_{2k-2-2n,\mathfrak{z}}$ with eigenvalue $(2k-2-n)(n+1)$.  

\end{theorem}
\begin{remark}
For $n<0$, the functions $z\mapsto R_{2k,\mathfrak{z}}^{-n-1}(H_{2k,N}(\mathfrak{z},z))$ from Lemma \ref{lem:fsplit} are constant multiples of the meromorphic parts of $\mathfrak{z}_2^{2-2k+n}\mathbb{P}_{2-2k,n,N}(z,\mathfrak{z})$.  
\end{remark}
\begin{proof}[Sketch of proof]
In order to show that $\mathfrak{z}_2^{2-2k+n}\mathbb{P}_{2-2k,n,N}(z,\mathfrak{z})$ is a polar harmonic Maass form as a function of $z$ and a polar Maass form as a function of $\mathfrak{z}$, we rewrite the Poincar\'e series in terms of functions defined in \cite{Fay}.  In particular, one can show that the $\mathbb{P}_{2-2k,n,N}$ are constant multiples of the Poincar\'e series $\mathfrak{z}_2^{k-1}y^{k-1}G_{k,1-k}^{-n,0}(\mathfrak{z},z)$ given in (44) of \cite{Fay}. Combining with the results in Theorem 2.1 of \cite{Fay} and rewriting Fay's differential operator $D_{1-k}$ in terms of $\Delta_{2-2k}$ yields that the functions $\mathfrak{z}_2^{2-2k+n}\mathbb{P}_{2-2k,n,N}(z,\mathfrak{z})$ are polar Maass forms in both variables.  
Furthermore, the eigenvalues in each variable are also explicitly given in Theorem 2.1 of \cite{Fay}.  Its principal part comes from the terms with $M\in \Gamma_{\mathfrak{z},N}$.  

It remains to prove \eqref{eqn:xiPoincn} and \eqref{eqn:DPoinc}.  The argument for both is similar, so we only give an idea of proof of \eqref{eqn:DPoinc}.  By (45) of \cite{Fay}, applying $R_{2-2k,z}^{2k-1}$ to $y^{k-1}G_{k,1-k}^{-n,0}(\mathfrak{z},z)$ yields a constant times $y^{-k}G_{k,1-k}^{-n,2k-1}(\mathfrak{z},z)$.  One can show that this is a constant multiple of Petersson's  $\Psi_{2k,n+1-2k,N}(z,\mathfrak{z})$.  Bol's identity then implies \eqref{eqn:DPoinc}. 
\end{proof}

We next discuss a splitting of weight $2k$ meromorphic cusp forms which is motivated by combining \eqref{eqn:xiPoincn} and \eqref{eqn:DPoinc}.  For the case $k=1$, which is omitted here, recall that weight $2$ meromorphic modular forms $f$ are in one-to-one correspondence with meromorphic differentials $f(z)dz$.  Traditionally, there is a splitting of meromorphic differentials into three kinds: the first kind comes from cusp forms, the second coincides under the bijection above with meromorphic forms that are orthogonal to cusp forms and have trivial residues, and the third kind corresponds to those with at most simple poles.  For weakly holomorphic forms, the space of forms yielding differentials of the third kind is precisely the space spanned by Eisenstein series.  Generalizing this to higher weight, the space of Eisenstein series is distinguished in a number of ways.  Concentrating on one of those, 
for each Eisenstein series 
there exists a weight $2-2k$ harmonic Maass form which maps to a constant multiple of 
it under both $\xi_{2-2k}$ and $D^{2k-1}$. One sees from \eqref{eqn:xiPoincn} and \eqref{eqn:DPoinc} that the forms $\Psi_{2k,-k,N}$ share this property, yielding a meromorphic analogue of the Eisenstein series.  Since Eisenstein series correlate with differentials of the third kind, it seems reasonable to view the functions $\Psi_{2k,-k,N}$ as the generalization to higher weight of meromorphic cusp forms corresponding to differentials of the third kind.   Based on the Poincar\'e series \eqref{eqn:weakPoinc}, the Eisenstein series split the space of weight $2k$ weakly holomorphic forms into three subspaces, those spanned by the Poincar\'e series with $m>0$, $m<0$, and $m=0$, respectively.  The space spanned by the Poincar\'e series with $m>0$ is precisely the space of cusp forms and extending to weight $2$ via analytic continuation, this  matches the splitting of meromorphic differentials.  Furthermore, if $\mathcal{F}\in H_{2-2k}^{\operatorname{cusp}}(N)$, then $D^{2k-1}(\mathcal{F})$ is in the space spanned by forms with $m<0$, so the differential operators naturally split the space.  The spaces spanned by forms with $m<0$ and $m>0$ are also orthogonal under a suitably regularized inner product.

Paralleling this for meromorphic cusp forms, one can also split the space into three pieces. Let $\SSnew_{2k}^{\mathfrak{z}}(N)$ denote the subspace of $\SSnew_{2k}(N)$ allowing poles at most at $\mathfrak{z}$.  We then write
$$
\SSnew_{2k}^{\mathfrak{z}}(N) = \mathbb{X}_{2k}^{\mathfrak{z}}(N)+\mathbb{D}_{2k}^{\mathfrak{z}}(N)+\mathbb{E}_{2k}^{\mathfrak{z}}(N),
$$
where $\mathbb{X}_{2k}^{\mathfrak{z}}(N)$ is spanned by those forms with poles of order at most $k-1$, $\mathbb{E}_{2k}^{\mathfrak{z}}(N)$ is spanned by the form $\Psi_{2k,-k,N}(z,\mathfrak{z})$, and $\mathbb{D}_{2k}^{\mathfrak{z}}(N)$ is spanned by $\Psi_{2k,n,N}(z,\mathfrak{z})$ with $n<-k$.  
One can show that 
if $\mathcal{F}\in\H_{2-2k}(N)$ satisfies $\xi_{2-2k}(\mathcal{F})\in \bigoplus_{\mathfrak{z}}\mathbb{X}_{2k}^{\mathfrak{z}}(N)$, then $\mathcal{D}^{2k-1}(\mathcal{F})\in \bigoplus_{\mathfrak{z}}\mathbb{D}_{2k}^{\mathfrak{z}}$.  
Hence, motivated by the differential operators and the fact that $\mathbb{E}_{2k}^{\mathfrak{z}}(N)$ parallels the Eisenstein series, one may ask whether, under a suitable regularized inner product,
$$
\mathbb{X}_{2k}^{\mathfrak{z}}\perp\mathbb{D}_{2k}^{\mathfrak{z}}.
$$
Jenkins and the authors are currently investigating this question using an inner product which is defined in Section \ref{sec:regularize}.

\subsection{Other representations of $\Psi_{2k,-1}(z,\mathfrak{z})$}
Note that the polar harmonic Maass form $F_z(\mathfrak{z}):= \mathfrak{z}_2^{2k-1}\Psi_{2k,-1}(z,\mathfrak{z})$ may be presented in a few other ways.  Firstly, since $\mathbb{P}_{2-2k,-1}(\mathfrak{z},z)$ has the same principal part as $F_{z}(\mathfrak{z})$ up to a constant multiple and its weight is negative, the two functions must agree up to a constant. Alternatively, a direct calculation, using (23) of \cite{Pe2} to flip the slash operator between the two variables, yields
$$
F_z(\mathfrak{z})=\frac{2k-1}{\left(4y\right)^{2k-1} }\mathbb{P}_{2-2k,-1}(\mathfrak{z},z).
$$
One may also explicitly realize $F_{z}(\mathfrak{z})$ in terms of the functions $H_{2k}$ via (see Proposition 3.2 of \cite{BKFourier}) 
\begin{equation*}
F_z(\mathfrak{z}) =\frac{\pi}{(-4)^{k-1}}\mathcal{H}_{2k}(z,\mathfrak{z}),
\end{equation*}
where
\begin{equation*}
\mathcal{H}_{2k}\left(z,\mathfrak{z}\right) := H_{2k}(z,\mathfrak{z}) + \sum_{r=0}^{2k-2}\frac{\left(2i\mathfrak{z}_2\right)^r }{r!}\frac{\partial^r}{\partial\overline{\mathfrak{z}}^r} H_{2k}\left(z,\overline{\mathfrak{z}}\right).
\end{equation*}
Furthermore, for $y>\max(\mathfrak{z}_2,1/\mathfrak{z}_2)$, 
(3.9) of \cite{BKFourier} states that $\mathcal{H}_{2k}(z,\mathfrak{z})$ may be written as the generating function 
\begin{equation*}
\mathcal{H}_{2k}\left(z,\mathfrak{z}\right)=-2\sum_{n=1}^{\infty}\mathcal{F}_{2-2k,n}(\mathfrak{z}) e^{2\pi i nz}
\end{equation*}

\section{Fourier coefficients of meromorphic modular forms}\label{sec:Fourier}

Before discussing the coefficients of weight $2-2k<0$ meromorphic cusp forms, we recall some of the history of the Fourier coefficients of weakly holomorphic modular forms.  Hardy and Ramanujan \cite{HR1,HR2} proved, using the Circle Method, that, as $n\to\infty$,
\begin{equation}\label{eqn:partitiongrowth}
p(n)\sim \frac{1}{4n\sqrt{3}} e^{\pi \sqrt{  \frac{2n}3}  }.
\end{equation}
Rademacher \cite{Rad} then perfected this method to derive the exact formula
\begin{equation}\label{Radformula}
p(n)= 2 \pi (24n-1)^{-\frac{3}{4}} \sum_{j =1}^{\infty}\frac{A_j(n)}{j} I_{\frac{3}{2}}\left( \frac{\pi\sqrt{24n-1}}{6j}\right).
\end{equation}
Here $I_{\ell}(x)$ is the $I$-Bessel function of order $\ell$ and $A_j(n)$ denotes a certain Kloosterman sum. These asymptotic formulas use the fact that the generating 
function
\rm
\begin{displaymath}
P(q):=\sum_{n=0}^{\infty}p(n)q^n=\prod_{n=1}^{\infty}(1-q^n)^{-1},
\end{displaymath}
is essentially a weight $-1/2$ weakly holomorphic modular form.  As mentioned in the introduction, Rademacher and Zuckerman  \cite{RZ,Zu1,Zu2} obtained similar formulas for all 
negative-weight weakly holomorphic modular forms.

We next consider Fourier coefficients of meromorphic cusp forms, starting with Ramanujan's conjectured formula for $1/E_4$ which was later proven by Bialek \cite{Bi}. 
In this case, we have, for $y>\sqrt{3}/2$, 
\begin{equation*}
 \frac{1}{E_4(z)} = \sum_{n=0}^{\infty} \beta_n e^{2\pi i  n z}
\end{equation*}
with 
\begin{equation}\label{eqn:1/E_4}
 \beta_n :=  \frac{6}{E_6(\rho)}\sum_{(\lambda)}\sum_{(c,d)} \frac{h_{(c,d)}(n)}{\lambda^3} e^{\frac{\pi n \sqrt{3}}{\lambda}}.
\end{equation}
Here $\rho:=e^{\frac{\pi i}{3}}$, $(c,d)$ runs over distinct solutions to $\lambda=c^2-cd+d^2$, where $\lambda=3^a \prod_{j=1}^{r}p_j^{a_j}$ with $a\in\{0,1\}, p_j$ denoting primes of the form $6m+1$, and $a_j \in \N_0$.  We do not  give the definition of distinct here, but the interested reader may find it after (2.2.3) of \cite{Bi}.  
Finally, we let $h_{(1,0)}(n):=(-1)^{n}/2$, $h_{(2,1)}(n):=1/2$, and for $\lambda \geq 7$
\begin{equation*}
 h_{(c,d)}(n):=   \cos\left( (ad+bc -2ac-2bd )\frac{\pi n}{\lambda}-6 \arctan\left(\frac{c\sqrt{3}}{2d-c}\right)\right),
\end{equation*}
where $a,b\in \Z$ are any choices for which $ad-bc=1$.  
It turns out that the Fourier coefficients of all of the specific meromorphic modular forms investigated by Ramanujan may be written as linear combinations of the series
\begin{equation}\label{Fklr}
F_{\kappa,j,r}(\mathfrak{z},z):= \mathfrak{z}_2^{-j}\sum_{m=0}^{\infty} \;\sideset{}{^*}\sum_{\mathfrak{b}\subseteq\mathcal{O}_{\Q(\mathfrak{z})}} \frac{C_{\kappa}\left(\mathfrak{b},m\right)}{N(\mathfrak{b})^{\frac{\kappa}{2}-j}} (4\pi m)^r e^{\frac{2\pi  m\mathfrak{z}_2}{N(\mathfrak{b})}} e^{2\pi i m z},
\end{equation}
where $\mathfrak{z}\in \{i,\rho\}$, $\mathfrak{b}$ runs over primitive ideals, $N(\mathfrak{b})$ is the norm, and $C_{\kappa}$ are certain functions on ideals which we next describe.  
For $\mathfrak{b}=(c\rho+d)\subset\mathcal{O}_{\Q(\rho)}$, we define 
\begin{equation*}
C_{6m}\left(\mathfrak{b},n\right) := 
\cos\left(\left(ad+bc-2ac-2bd\right) \frac{\pi n}{N(\mathfrak{b})}-6m\arctan\left(\frac{c\sqrt{3}}{2d-c}\right)\right),
\end{equation*}
and we set $C_{m}(\mathfrak{b},n):=0$ if $6\nmid m$.
 Similarly, for $\mathfrak{b}=(ci+d)\subseteq\mathcal{O}_{\Q(i)}$, we let
\begin{equation*}
C_{4m}\left(\mathfrak{b},n\right):=\cos\left(\left(ac+bd\right)\frac{2\pi n}{N(\mathfrak{b})}+4m\arctan\left(\frac{c}{d}\right)\right)
\end{equation*}
and $C_{m}(\mathfrak{b},n):=0$ if $4\nmid m$.

After analyzing a number of examples, Berndt, Bialek, and Yee \cite{BeBiYe} recognized certain similarities between the coefficients.  
The pattern they noticed is much more general.  Indeed, every negative weight meromorphic cusp whose only pole in the standard fundamental domain $\mathcal{F}$ occurs at $z_0\in\{i,\rho\}$ has an 
expansion as a linear combination of the functions defined in \eqref{Fklr}.  
\begin{theorem}[Theorem 1.2 of \cite{BKFourier}]\label{Fouriercoefficients}
If $z_0\in\{i,\rho\}$ and $a_n\in\C$ are chosen such that 
\begin{equation}\label{eqn:fsplit}
f(z) = \sum_{n=0}^{n_0} a_{n} R_{2k,\mathfrak{z}}^{n}\left[H_{2k}\left(\mathfrak{z},z\right)\right]_{\mathfrak{z}=z_{0}}
\in \SSnew_{2-2k},
\end{equation}
then, for $y>y_0$, we have the Fourier expansion 
$$
f(z) = 2\omega_{z_0}\sum_{n=0}^{n_0} a_{n} \sum_{j=0}^{n} \frac{(2k+n-1)!}{(2k+n-1-j)!}\binom{n}{j} F_{2k+2n,j,n-j}(z_0,z).
$$
\end{theorem}
\begin{remarks}
\noindent
\begin{enumerate}[leftmargin=*]
\item
By Lemma \ref{lem:fsplit}, every element of $\SSnew_{2-2k}$ which has a unique pole in $\SL_2(\Z)\backslash\H$ also has 
an expansion of the type in \eqref{eqn:fsplit}.
\item
 A more general formula for meromorphic cusp forms with poles at arbitrary points is given in Theorem 4.1 of \cite{BKFourier}. 
\item
It was shown in Theorem 1.3 of \cite{BKFourier} that \begin{it}quasi-meromorphic cusp forms\end{it}, which are products of powers of the weight 2 quasi-modular Eisenstein series $E_2$ times meromorphic cusp forms, also have Fourier expansions of this shape. 
\end{enumerate}
\end{remarks}
\begin{proof}[Sketch of proof]
The proof of Theorem \ref{Fouriercoefficients} has essentially 3 steps.  Firstly, one constructs  a basis for the space $\H_{2-2k}^{\operatorname{cusp}}$ by applying repeated raising in $\mathfrak{z}$ to $\mathcal{H}_{2k}(\mathfrak{z},z)$.  
One then determines that the meromorphic parts of the resulting basis elements are the forms from Lemma \ref{lem:fsplit}.  
Noting that $f$ is meromorphic, the problem is hence reduced to computing the Fourier coefficients of $R_{2k,\mathfrak{z}}^{n}\left[H_{2k}\left(\mathfrak{z},z\right)\right]_{\mathfrak{z}=z_{0}}$.  By Proposition 4.5 of \cite{BKFourier}, 
we may rewrite
\begin{equation*}
R_{2k,\mathfrak{z}}^{n}\left(H_{2k}(\mathfrak{z},z)\right)=\sum_{j=0}^{n} \frac{(2k+n-1)!}{(2k+n-1-j)!}\binom{n}{j} (-2i)^{n-j} \frac{\partial^{n-j}}{\partial z^{n-j}} H_{2k+2n,j}(\mathfrak{z},z),
\end{equation*}
where
\begin{equation*}
H_{2k,j}(\mathfrak{z},z):=
2
\sum_{M\in \Gamma_{\infty}\backslash \SL_2(\Z)}\frac{\mathfrak{z}_2^{-j}}{1-e^{2\pi i (z-\mathfrak{z})}}\bigg|_{2k,\mathfrak{z}} M.
\end{equation*}
The coefficients of $H_{2k,j}$ were computed in Theorem 3.1 of \cite{BKRamanujan}, which is the last step.
\end{proof}

To use the result in explicit examples, one only needs to compute the principal part of a given meromorphic cusp form to determine its representation in the form \eqref{eqn:fsplit}.  This is carried out for $1/E_6^4$ in Theorem 6.2 of \cite{BKFourier}.  
\begin{remark}
There are also results for the Fourier coefficients of meromorphic modular forms of higher level if the pole order is small \cite{Zap}. 
\end{remark}

\section{Theta lifts and inner products of meromorphic modular forms}\label{sec:cycle}
In this section, we investigate lifts between forms of weights $k+1/2$ and $2k$ as well as lifts from forms of weight $3/2-k$ to $2-2k$, where $k\in\N_{\geq 2}$. 
 
 \subsection{Regularized inner products and theta lifts}\label{sec:regularize}

Although \eqref{eqn:innerclassical}
 generally diverges if $g$ or $h$ have poles or grow towards the cusps, for a number of applications it is useful to take such inner products.  Petersson \cite{Pe2} appears to be the first to consider this problem, defining a regularized inner product $\left<g,h\right>$ via the Cauchy principal value of the naive definition,
 essentially cutting out small balls of hyperbolic radius $\varepsilon >0$ around each pole and then taking $\varepsilon\to 0$.
If  the poles are at the cusps, his construction was rediscovered and extended by Borcherds \cite{Bo1} and Harvey and Moore \cite{HM}, and then used by Bruinier \cite{Bruinier} and others to define theta lifts of weakly holomorphic modular forms.   We are particularly interested in theta lifts coming from the Shintani theta kernel 
\begin{equation*}
\Theta_k(z,\tau):=y^{-2k}v^{\frac{1}{2}}\sum_{D\in \Z}\sum_{Q\in\mathcal{Q}_D}Q(z,1)^k e^{-4\pi Q_{z}^2 v}e^{2\pi i D\tau}.
\end{equation*}

The function $z\mapsto \Theta_k(-\overline{z},\tau)$ has weight $2k$ and $\tau\mapsto\Theta_k(z,\tau)$ has weight $k+1/2$.  Hence taking the inner product in one variable yields a lift between forms satisfying modularity in integral and half-integral weights.  In particular, for a weight $k+1/2$ function $H$, we define the \begin{it}theta lift of $H$\end{it},
\[
\Phi_k(H)(z):= \left<H,\Theta_k\left(z,\cdot\right)\right>.
\]
The functions $f_Q$ were realized as theta lifts of weakly holomorphic modular forms by 
von Pippich and the authors \cite{BKCycle}. In order to obtain such a lift, one may apply $\Phi_k$ to 
\begin{equation*}
g_{k,-D}:=\frac{(4\pi)^{k}}{(k-1)!}D^{\frac{k}{2}}P_{k+\frac{1}{2},-D}.
\end{equation*}

The lift is evaluated by a standard unfolding argument, yielding the following.
\begin{theorem}[Theorem 1.1 of \cite{BKCycle}]\label{thm:liftfkd}
For every discriminant $-D<0$, we have 
\begin{equation*}
\Phi_k\left(g_{k,-D}\right)=f_{k,-D}.
\end{equation*}
\end{theorem}
A closely-related theta lift in negative weight is constructed via the theta kernel
\begin{equation*}
\Theta_{1-k}^*(z,\tau):=v^k \sum_{D\in \Z}\sum_{Q\in \mathcal{Q}_D} Q_{z}Q(z,1)^{k-1} e^{-\frac{4\pi |Q(z,1)|^2 v}{y^2}}  e^{-2\pi i D\tau}.
\end{equation*}
The function $z\mapsto \Theta_{1-k}^*(z,\tau)$ has weight $2-2k$ 
and
\rm
 $\tau\mapsto \Theta_{1-k}^*(z,\tau)$ has weight $3/2-k$ (see Theorem 4.1 of \cite{Bo1} for a general treatment and Proposition 3.2 (2) of \cite{BKM} for this case).
\rm
For a function $H$ transforming of weight $3/2-k$, we thus define the theta lift
$$
\Phi_{1-k}^* (H)(z):=\left<H,\Theta_{1-k}^*\left(-\overline{z},\cdot\right)\right>.
$$
We obtain $\mathcal{G}_Q$, defined in \eqref{eqn:Gdef}, by applying $\Phi_{1-k}^*$ to the harmonic Maass form  
$$
\mathcal{P}_{1-k,-D}:=\frac{12\sqrt{\pi}\Gamma\left(k+\frac{1}{2}\right)}{D^{k-\frac{1}{2}}(k-1)!(2k-1)} F_{2-2k,-D}.
$$

\begin{theorem}[Theorem 1.3 of \cite{BKCycle}]\label{thm:thetalift}
For $-D<0$ a discriminant, we have 
$$
\Phi_{1-k}^*\left(\mathcal{P}_{1-k,-D}\right) = \sum_{Q\in\mathcal{Q}_{-D}/\SL_2(\Z)} \mathcal{G}_Q.
$$
\end{theorem}

The functions $\mathcal{G}_Q$ are furthermore related to $f_{Q}$ via $\xi_{2-2k}$, as given in the following theorem.

\begin{theorem}[Theorem 1.2 of \cite{BKCycle}]\label{thm:Gpolar}
The functions $\mathcal{G}_Q$ are weight $2-2k$ polar harmonic Maass forms whose only  singularities in $\H\setminus\SL_2(\Z)$ occur at $\tau_Q$.  Furthermore, we have 
\begin{align}
\label{eqn:xiG}\xi_{2-2k}\left(\mathcal{G}_{Q}\right) &= f_{Q},\\
\mathcal{D}^{2k-1}\left(\mathcal{G}_{Q}\right)&= -\frac{(2k-2)!}{(4\pi)^{2k-1}} f_{Q}. \nonumber
\end{align}
\end{theorem}
\begin{remark}
It is rather unusual for the images of a harmonic function to agree (up to a constant) under both operators.  For example, although the Eisenstein series has this property, 
it is impossible if the images are both cusp forms.
\rm
\end{remark}
\begin{proof}[Sketch of proof]
One obtains the modularity of $\mathcal{G}_{Q}$ by writing it as a Poincar\'e series.  
The functions $f_Q$ also have presentations 
as Poincar\'e series, and both of the differential operators map the terms of $\mathcal{G}_Q$ to a constant multiple of the summands in $f_Q$.  Similarly to Bengoechea's \cite{Bengoechea} case for $f_Q$, one of the main steps is to show absolute and locally uniform convergence of $\mathcal{G}_Q$ 
so that one may apply differential operators termwise.

\end{proof}

\subsection{Inner products between meromorphic modular forms.}
The functions $\mathcal{G}_Q$ reappear when taking inner products of $f_{Q}$ against meromorphic cusp forms with simple poles. Since Petersson's regularization still diverges sometimes, one requires a further regularization.  Roughly speaking, the integrand in \eqref{eqn:innerclassical} is multiplied by an SL$_2(\mathbb{Z})$-invariant function $H_s(\tau)$ which removes the poles of the integrand. We then take the constant term of the Laurent expansion around $s=0$ to be our regularization.  To be more precise, let $[z_1],\dots, [z_r]\in \operatorname{PSL}_2(\Z)\backslash\H$ be the distinct $\SL_2(\Z)$-equivalence classes of all of the poles of $g$ and $h$ and choose a fundamental domain $\mathcal{F}^*$ such that all $z_{\ell}$ lie in the interior of $\Gamma_{z_{\ell}}\mathcal{F}^*$. Then let
\begin{equation*}
\left<g,h\right>:=
 \operatorname{CT}_{s=0}\left(\int_{\SL_2(\Z)\backslash \H} g(z) H_s(z) \overline{h(z)} y^{2k}\frac{dx dy}{y^2}\right),
\end{equation*}
where 
$$
H_s (z) = H_{s_1,\dots, s_r, z_1, \dots, z_r} (z) := \prod_{\ell=1}^r h_{s_{\ell},z_{\ell}} (z).
$$
Here
$$
h_{s_{\ell},z_{\ell}}(z):=r_{z_{\ell}}^{2s_{\ell}}(Mz),
$$
with $M \in \SL_2(\Z)$ chosen such that $Mz\in\mathcal{F}^*$. Moreover $\operatorname{CT}_{s=0}$ 
denotes the constant term in the Laurent expansion around $s_1=s_2=\cdots =s_{r}=0$ of the meromorphic continuation (if existent).

This regularization was used to compute the inner product between $f_{Q}$ and $f\in\SSnew_{2k}$.  The results in the case that the poles of $f$ are all simple are collected below.

\begin{theorem}[Theorems 1.7 (2) and 1.8 (2) of \cite{BKCycle}]\label{thm:innersimple}
Suppose that $f\in\SSnew_{2k}$ has its poles in $\SL_2(\Z)\backslash\H$ at $[z_1],\dots, [z_r]$ with $[z_{\ell}]\neq [z_j]$ for $\ell\neq j$ and that all poles are simple.
\noindent

\noindent
\begin{enumerate}[leftmargin=*]
\item[\rm(1)]
If $[z_{\ell}]\neq [\tau_Q]$ for $1\leq \ell\leq r$, then  
$$
\left<f,f_{Q}\right> =2\pi i\sum_{\ell=1}^{r}\frac{1}{\omega_{z_{\ell}}}  \mathcal{G}_{Q}(z_{\ell}) \text{\emph{Res}}_{z=z_{\ell}} f(z).
$$
\item[\rm(2)]
If $z_r=\tau_Q$, then 
\begin{multline*}
\left<f,f_{Q}\right> =2\pi i\sum_{\ell=1}^{r-1}\frac{1}{\omega_{z_{\ell}}}  \mathcal{G}_{Q}(z_{\ell}) \text{\emph{Res}}_{z=z_{\ell}} f(z)\\+\frac{2\pi i}{\omega_{\tau_{Q}}}D^{\frac{1-k}{2}}\Res_{z=\tau_{Q}} f(z) \sum_{\mathcal{Q}\in [Q]\setminus \{Q\}} \mathcal{Q}(\tau_Q,1)^{k-1}\int_0^{\arctanh\left(\frac{\sqrt{D}}{\mathcal{Q}_{\tau_Q}}\right)} \sinh^{2k-2}(\theta) d\theta.
\end{multline*}
\end{enumerate}
\end{theorem}
\begin{remarks}
\noindent

\noindent
\begin{enumerate}[leftmargin=*]
\item
A more general version of Theorem \ref{thm:innersimple}, allowing higher order poles of $f$, is given in Theorem 1.8 of \cite{BKCycle}.   The residue of $f$ is replaced with elliptic coefficients in the principal parts of $f$ and a more general family of polar Maass forms appears in place of $\mathcal{G}_Q$.
\item
Zemel has obtained an interesting vector-valued 
generalization of Theorem \ref{thm:innersimple} in Theorem 10.3 of \cite{Zemel}.   

\end{enumerate}
\end{remarks}
\begin{proof}[Sketch of proof of Theorem \ref{thm:innersimple}]
(1)  First recall that, by \eqref{eqn:xiG}, $\mathcal{G}_Q$ is a preimage of $f_Q$ under $\xi_{2-2k}$.  We use Stokes Theorem and follow the proof of Proposition \ref{expansionprod}, except that we replace the cusp form $g$ by a meromorphic cusp form.  From this we obtain a formula for $\left<f,f_{Q}\right>$ resembling \eqref{eqn:gFpair}, except that we also have a contribution from the principal part of the elliptic expansions of $f$.  In particular, if all poles $[z_1],\dots [z_r]\in \SL_2(\Z)\backslash\H$ of $f$ are simple, then one gets the product of the $0$th elliptic coefficient of $\mathcal{G}_{Q,z_{\ell}}^+$ times the $(-1)$th coefficient in the elliptic expansion of $f$ around $z=z_{\ell}$.  One concludes the claim by realizing that the $0$th coefficient of $\mathcal{G}_{Q,z_{\ell}}^+$ is essentially $\mathcal{G}_Q(z_{\ell})$ and  the $(-1)$th coefficient in the elliptic expansion of $f$ roughly equals its residue.  

\noindent
(2)  
The argument is essentially the same, but there is an additional difficulty which arises since  $z_r=\tau_Q$, so $\mathcal{G}_Q$ has a 
singularity at $z_r$.  In this case, one determines that the $0$th elliptic coefficient of $\mathcal{G}^+_{Q,\tau_Q}$ is the evaluation at $z=\tau_Q$ of $\mathcal{G}_Q$ minus the $\mathcal{Q}=Q$ term.
\end{proof}

\subsection{Relations with higher Green's functions}

The functions $\mathcal{G}_Q$ are also closely related to evaluations of higher Green's functions at CM-points.
\begin{proposition}\label{prop:GQGreens}
We have 
$$
\mathcal{G}_Q(z) = \frac{1}{2^k\omega_{\tau_Q}(k-1)!} y^{2k-2}\overline{R_{0}^{k-1}\left(G_k^{\H/\SL_2(\Z)}\!\left(z,\tau_Q\right)\right)}.
$$
\end{proposition}
\begin{proof}[Sketch of proof of Proposition \ref{prop:GQGreens}]
One begins by directly computing a formula for the higher Green's functions raised and lowered in either variable 
(see also page 24 of \cite{Mellit}).  This formula is written in terms of Gauss's hypergeometric ${_2F_1}$, which we then show essentially matches the integral appearing in $\mathcal{G}_Q$.  Finally, we write $Q(z,1)$ in terms of its roots $\tau_Q$ and $\overline{\tau_Q}$ to finish comparing the two sides.

\end{proof}

Theorem \ref{thm:innersimple} together with Proposition \ref{prop:GQGreens} yields a connection between inner products and higher Green's functions.  However, if we replace the form $f$ from Theorem \ref{thm:innersimple} with another $f_{\mathcal{Q}}$, we get an even more direct relation between these two objects.
\begin{theorem}\label{thm:innerGreens}
For $Q\in\mathcal{Q}_{-D_1}$ and $\mathcal{Q}\in Q_{-D_2}$ ($-D_1,-D_2<0$ discriminants) with $[\tau_Q]\neq [\tau_{\mathcal{Q}}]$, we have
$$
\left<f_{\mathcal{Q}},f_{Q}\right>=-\frac{\pi(-4)^{1-k}}{(2k-1)\beta(k,k)} \frac{G_k^{\H/\SL_2(\Z)}\!\left(\tau_{\mathcal{Q}},\tau_Q\right)}{\omega_{\tau_{\mathcal{Q}}}\omega_{\tau_Q}}.
$$
\end{theorem}
\begin{proof}[Sketch of proof of Theorem \ref{thm:innerGreens}]
The first part of the proof closely follows the application of Stokes Theorem in the proof of Theorem \ref{thm:innersimple} (1).  If 
the function $f$ from Theorem \ref{thm:innersimple} 
is chosen to have  principal part  $X_{\mathfrak{z}}^{-n}(z)$ instead of a  simple pole, then the function $\mathcal{G}_Q$ is replaced with a constant multiple of $R_{2-2k}^{n-1}(\mathcal{G}_Q)$.  Applying raising $k-1$ times to Proposition \ref{prop:GQGreens}, in particular, yields a constant multiple of $G_k^{\H/\SL_2(\Z)}$.  Since the pole of $f=f_{\mathcal{Q}}$ occurs at $\tau_{\mathcal{Q}}$, we obtain the higher Green's function evaluated at these CM-points.
\end{proof}
\begin{remark}
By replacing $f_{\mathcal{Q}}$ and $f_Q$ with slightly more general functions $\Psi_{2k,-k}(\cdot,z_j)$, defined in \eqref{eqn:Psidef}, having poles at arbitrary $z_1,z_2\in\H$, one can 
obtain a relation between the inner product and $G_k^{\H/\SL_2(\Z)}(z_1,z_2)$.
\end{remark}
If $k=1$, which is excluded here, Gross and Zagier (see Proposition 2.22 in Section II of \cite{GZ}) related the Green's function at CM-points with the 
infinite part of the height pairing on degree $0$ divisors.  This has been generalized 
to higher $k$ by Zhang \cite{Zhang}, who obtained an identity between the 
infinite part of the height pairing of CM-cycles and the higher Green's functions at CM-points.  Combining Zhang's result with Theorem \ref{thm:innerGreens} immediately yields a relation between the inner product and the height pairing of CM-cycles.  Denoting by $S_k(\tau_Q)$ the CM-cycle associated to $\tau_Q$ 
that was constructed in Section 0.1 of \cite{Zhang}, and using $\left<\cdot,\cdot\right>_h$ to denote the infinite part of the height pairing, the following is an immediate consequence of Theorem \ref{thm:innerGreens} and Proposition 4.1.2 of \cite{Zhang}.   
\begin{corollary}\label{cor:innerheight} 
For $Q\in\mathcal{Q}_{-D_1}$ and $\mathcal{Q}\in Q_{-D_2}$ ($-D_1,-D_2<0$ discriminants) with $[\tau_Q]\neq [\tau_{\mathcal{Q}}]$, we have
$$
\left<f_{\mathcal{Q}},f_{Q}\right>=-\frac{2\pi(-4)^{1-k}}{(2k-1)\beta(k,k)} \frac{\left<S_k\left(\tau_{\mathcal{Q}}\right),S_k\left(\tau_{Q}\right)\right>_h}{\omega_{\tau_{\mathcal{Q}}}\omega_{\tau_Q}}.
$$
\end{corollary}

\section{Future directions}\label{sec:future}
We conclude the paper by discussing some possible future directions that research involving polar harmonic Maass forms could take.  

\noindent
\begin{enumerate}[leftmargin=*]
\item
It would be interesting to look for a meromorphic cusp form or polar harmonic Maass form which encodes combinatorial information and investigate the application of Theorem \ref{Fouriercoefficients} to this form.  
\item 
Another natural function to study is $1/E_2$.  Although it is neither modular nor quasi-modular, its connection to the quasi-modular Eisenstein series $E_2$ leads one to ask what its Fourier expansion looks like and whether it belongs to a related space. Sebbar and Sebbar \cite{SebbarSebbar} investigated this function from another perspective.
\item   In Theorem \ref{thm:thetalift} there is a technical restriction because for $n\geq k$ the naive definition of $\Phi_{n-k}^*$ has poles at every CM-point in $\H$.  Since $\mathcal{F}_{2k,Q}=f_Q$ is also a theta lift (using the theta kernel $\Phi_k$), it would be interesting to see if the other functions $\sum_{\mathcal{Q}\in[Q]} \mathcal{F}_{n,Q}$ are theta lifts for all $n\in\N$.   Steffen L\"obrich is currently investigating this question.

\item Since the Poincar\'e series $\mathbb{P}_{2-2k,n,N}$ are preimages of the functions $\Psi_{2k,-n-1,N}$ by Theorem \ref{thm:Poincmain}, following the proof of Theorem \ref{thm:innersimple}, one could relate general inner products of meromorphic cusp forms to the functions $\mathbb{P}_{2-2k,n,N}$ evaluated at points in $\H$.
\item
 Recall the connection between the generating function of Zagier's $f_{k,\delta}\ (\delta>0)$ and cycle integrals, which may be considered real quadratic traces.  One may wonder whether there is a connection between CM-traces and the generating function of the $f_{k,-D}$.  The naive generating function diverges, but it would be interesting to find a natural regularization.
\item In Conjecture 4.4 of \cite{GZ}, Gross and Zagier take linear combinations of $G_{k}^{\H/\SL_2(\Z)}$ acted on by Hecke operators and conjecture that these linear combinations evaluated at CM-points are essentially logarithms of algebraic numbers whenever the linear combinations satisfy certain relations.  These relations are determined by linear equations defined by the Fourier coefficients of weight $2k$ cusp forms.  
Given the connection between $G_{k}^{\H/\SL_2(\Z)}(z,\tau_Q)$ and the weight $2-2k$ polar harmonic Maass form $\mathcal{G}_Q(z)$ from Proposition \ref{prop:GQGreens}, it might be interesting to understand their condition in the language of polar harmonic Maass forms.  On the one hand, it might carve out a natural subspace of weight $2k$ meromorphic modular forms (corresponding to the image of those forms satisfying these conditions) which may satisfy other interesting properties.  On the other side, by applying the theory of harmonic Maass forms, one may be able to loosen the conditions and investigate what happens for general linear combinations.  
\end{enumerate}


\begin{thebibliography}{99}

\bibitem{AGOR} C. Alfes, M. Griffin, K. Ono, and L. Rolen, \begin{it}Weierstrass Mock Modular forms and elliptic curves\end{it}, Research in Number Theory \textbf{1:24} (2015), 1--31.
\bibitem{Andrews}G. Andrews, \begin{it}On the theorems of Watson and Dragonette for Ramanujan's mock theta functions\end{it}, Amer. J. Math. \textbf{88} (1966), 454--490.
\bibitem{Bengoechea} P. Bengoechea, \begin{it}Corps quadratiques et formes modulaires,\end{it} Ph.D. thesis, 2013.
\bibitem{Bi}P. Bialek, \begin{it}Ramanujan's formulas for the coefficients in the power series expansions of certain modular forms\end{it}, Ph. D. thesis, University of Illinois at Urbana--Champaign, 1995.
\bibitem{BeBiYe} B. Berndt, P. Bialek, and A. Yee, \begin{it}Formulas of Ramanujan for the power series coefficients of certain quotients of Eisenstein series\end{it}, Int. Math. Res. Not. \textbf{2002} (2002), 1077--1109.
\bibitem{Bo1}R. Borcherds, \begin{it}Automorphic forms with singularities on Grassmannians\end{it}, Inv. Math. \textbf{132} (1998), 491--562. 
\bibitem{BJK} K. Bringmann, P. Jenkins, and B. Kane, \begin{it}Meromorphic modular forms and the $D^{k-1}$-operator\end{it}, in preparation.
\bibitem{BKRamanujan} K. Bringmann and B. Kane, \begin{it} Ramanujan and coefficients of meromorphic modular forms\end{it}, J. Math. Pures Appl., accepted for publication.
\bibitem{BKWeight0} K. Bringmann and B. Kane, \begin{it}A problem of Petersson about weight $0$ meromorphic modular forms\end{it}, submitted for publication.
\bibitem{BKFourier} K. Bringmann and B. Kane, \begin{it}Fourier coefficients of meromorphic modular forms and a question of Petersson\end{it}, submitted for publication.
\bibitem{BKW} K. Bringmann, B. Kane,  W. Kohnen, \begin{it}Locally harmonic Maass forms and the kernel of the Shintani lift\end{it}, Int. Math. Res. Not. \textbf{2015} (2015), 3185--3224.
\bibitem{BKCycle} K. Bringmann, B. Kane, and A. von Pippich, \begin{it}Cycle integrals of meromorphic modular forms and CM-values of automorphic forms\end{it}, in preparation.
\bibitem{BKM}K. Bringmann, B. Kane, and M. Viazovska, \begin{it}Theta lifts and local Maass forms\end{it}, Math. Res. Lett. \textbf{20} (2013), 213--234.
\bibitem{BOMockTheta} K. Bringmann and K. Ono, \begin{it}The $f(q)$ mock theta function conjecture and partition ranks\end{it}, Invent. Math. \textbf{165} (2006), 243--266.
\bibitem{BO2} K. Bringmann and K. Ono, \begin{it}Coefficients of harmonic Maass forms\end{it}, Proceedings of the 2008 University of Florida Conference on Partitions, $q$-series, and modular forms.
\bibitem{Bruinier} J. Bruinier, \begin{it}Borcherds products on $O(2,l)$ and Chern classes of Heegner divisors\end{it} Lecture Notes in Math. \textbf{1780}, 2002.
\bibitem{BF} J. Bruinier and J. Funke, \begin{it}On two geometric theta lifts\end{it}, Duke Math. J. \textbf{125} (2004), no.\ 1, 45--90.
\bibitem{BruinierOno} J. Bruinier and K. Ono, \begin{it}Algebraic formulas for the coefficients of half-integral weight harmonic weak Maass forms\end{it}, Adv. Math. \textbf{246} (2013), 198--219.
\bibitem{BOR}J. Bruinier, K. Ono, and R. Rhoades, \begin{it}Differential operators for harmonic weak Maass forms and the vanishing of Hecke eigenvalues\end{it}, Math. Ann. \textbf{342} (2008), 673--693.
\bibitem{Dragonette}L. Dragonette, \begin{it}Some asymptotic formulae for the mock theta series of Ramanujan\end{it}, Trans. Amer. Math. Soc. \textbf{72} (1952), 474--500.
\bibitem{Fay} J. Fay, \begin{it}Fourier coefficients of the resolvent for a Fuchsian group\end{it}, J. reine und angew. Math. \begin{bf}293--294\end{bf} (1977), 143--203.
\bibitem{GZ} B. Gross and D. Zagier, \begin{it}Heegner points and derivatives of $L$-series\end{it}, Invent. Math. \textbf{84} (1986), 225--320.
 \bibitem{HR1} G.  Hardy and S. Ramanujan, \begin{it}Une formule asymptotique pour le nombre des partitions de $n$\end{it},  Collected papers of Srinivasa Ramanujan,  239--241, AMS Chelsea Publ., Providence, RI, 2000.
\bibitem{HR2} G. Hardy and S. Ramanujan, \begin{it}Asymptotic formulae in combinatory analysis\end{it}, Proc. London Math. Soc. \textbf{16} (1917), in Collected papers of Srinivasa Ramanujan, AMS Chelsea Publ, Providence, RI, 2000, 244.
\bibitem{HR3} G. Hardy and S. Ramanujan, \begin{it}On the coefficients in the expansions of certain modular functions\end{it}, Proc. Royal Soc. A \textbf{95} (1918), 144--155.
\bibitem{HM} J. Harvey and G. Moore, \begin{it} Algebras, BPS states, and strings\end{it}, Nuclear Phys. B \textbf{463} (1996), 315--368.
\bibitem{Hejhal}D. Hejhal, \begin{it}The Selberg trace formula for $\operatorname{PSL}(2,\R)$ (Volume 2)\end{it} Lecture Notes in Math. \textbf{1001}, 1983. 
\bibitem{Hoevel} M. H\"ovel, \begin{it}Automorphe Formen mit Singularitaten auf dem hyperbolischen Raum\end{it}, Ph.D. thesis, 2012.
\bibitem{HondaKaneko}Y. Honda and M. Kaneko, \begin{it}On Fourier coefficients of some meromorphic modular forms\end{it}, Bull. Korean Math. Soc. \textbf{49} (2012), 1349--1357.
\bibitem{ImOs}\"O. Imamo$\overline{\text{g}}$lu and C. O'Sullivan, \begin{it}Parabolic, hyperbolic, and elliptic Poincar\'e series\end{it}, Acta Arith. \textbf{139} (2009), 199--228.
\bibitem{Knopp}M. Knopp, \begin{it}Construction of automorphic forms on $H$-Groups and supplementary Fourier series\end{it}, Trans. Amer. Math. Soc. \begin{bf}103\end{bf} (1962), 168--188.  Corrigendum:  Trans. Amer. Math. Soc. \begin{bf}106\end{bf} (1963), 341--345.
\bibitem{KohnenFourier} W. Kohnen, \begin{it}Fourier coefficients of modular forms of half-integral weight\end{it}, Math. Ann. \textbf{271} (1985), 237--268.
\bibitem{KohnenZagier} W. Kohnen and D. Zagier, \begin{it}Values of $L$-series of modular forms at the center of the critical strip,\end{it} Invent. Math. \textbf{64} (1981), 175--198.
\bibitem{Kramer}D. Kramer, \begin{it}Applications of Gauss's theory of reduced binary quadratic forms to zeta functions and modular forms,\end{it} Ph.D. thesis, University of Maryland, 1983.
\bibitem{Masri}R. Masri, \begin{it}Fourier coefficients of harmonic weak Maass forms and the partition function\end{it}, Amer. J. Math. \textbf{137} (2015), 1061--1097.
\bibitem{Mellit} A. Mellit, \begin{it}Higher Green's functions for modular forms\end{it}, Ph. D. thesis, Rheinische Friedrich-Wilhelms-Universit\"at Bonn, 2008.
\bibitem{Niebur} D. Niebur, \begin{it}A class of nonanalytic automorphic functions\end{it}, Nagoya Math. J. \textbf{52} (1973), 133--145.
\bibitem{Pe3} H. Petersson, \begin{it}Einheitliche Begr\"undung der Vollst\"andigkeitss\"atze f\"ur die Poincar\'eschen Reihen von reeler Dimension bei beliebigen Grenzkreisgruppen von erster Art\end{it}, Abh. Math. Sem. Hansischen Univ. \textbf{14} (1941), 22--60.
\bibitem{Pe1} H. Petersson, \begin{it}Konstruktion der Modulformen und der zu gewissen Grenzkreisgruppen geh\"origen automorphen Formen von positiver reeller Dimension und die vollst\"andige Bestimmung ihrer Fourierkoeffzienten\end{it}, S.-B. Heidelberger Akad. Wiss. Math. Nat. Kl. (1950), 415--474.
\bibitem{Pe2} H. Petersson, \begin{it}\"Uber automorphe Orthogonalfunktionen und die Konstruktion der automorphen Formen von positiver reeller Dimension\end{it}, Math. Ann. \textbf{127} (1954), 33--81.
\bibitem{PoincareSeries}H. Poincar\'e, \begin{it}Fonctions modulaires et fonctions fuchsiennes,\end{it} Oeuvres \textbf{2} (1911), 592--618.
\bibitem{RaLost} S. Ramanujan, \begin{it}The lost notebook and other unpublished paper\end{it}, Narosa, New Delhi, 1988.
\bibitem{Rad} H. Rademacher, {\it On the expansion of the partition function in a series}, Ann. of Math. \textbf{44} (1943), 416--422.
\bibitem{RZ} H. Rademacher and H.  Zuckerman,
{\it On the Fourier coefficients of certain modular forms of positive dimension},  Ann. of Math.   \textbf{39} (1938),  433--462.
\bibitem{SebbarSebbar} A. Sebbar and A. Sebbar, \begin{it}Equivariant functions and integrals of elliptic functions\end{it}, Geom. Dedicata \textbf{160} (2012), 373--414.
\bibitem{ZagierRQ} D. Zagier, \begin{it}Modular forms associated to real quadratic fields,\end{it} Invent. Math. \begin{bf}30\end{bf} (1975), 1--46.
\bibitem{ZagierBourbaki} D. Zagier, \begin{it}Ramanujan's mock theta functions and their applications [d'apr\'es Zwegers and Bringmann--Ono]\end{it}, S\'eminaire Bourbaki, Ast\'erisque \textbf{326} (2009), 143--164.
\bibitem{Zap} J.M. Zapata Rolon, \begin{it} On the Fourier coefficients of meromorphic modular forms of higher level\end{it}, in preparation. 
\bibitem{Zemel} S. Zemel, \begin{it}Regularized pairings of meromorphic modular forms and theta lifts\end{it}, J. Number Theory \textbf{162} (2016), 275--311.
\bibitem{Zhang} S. Zhang, \begin{it}Heights of Heegner cycles and derivatives of $L$-series\end{it}, Invent. Math. \textbf{130} (1997), 99-152.
\bibitem{Zu1} H.  Zuckerman, {\it On the coefficients of certain modular forms belonging to subgroups of the modular group}, Trans. Amer. Math. Soc. \textbf{45} (1939), 298--321.
\bibitem{Zu2} H.  Zuckerman, {\it On the expansions of certain modular forms of positive dimension}, Amer. J. Math. \textbf{62} (1940),  127--152.
\bibitem{Zwegers}S. Zwegers, \begin{it}Mock theta functions\end{it}, Ph.D. thesis, Utrecht Universiteit, 2002.


\end{thebibliography}
\end{document}